\documentclass[12.5pt, reqno]{amsart}
\usepackage[all]{xy}
\usepackage{amsfonts,amsmath,amssymb,amscd}
\usepackage[]{hyperref}
\usepackage{amsmath}
\usepackage{tabularx,ragged2e,booktabs,caption}

\usepackage{enumerate}
\usepackage{amsmath,amsthm,amscd,amsfonts,amssymb,mathrsfs,graphicx}
\usepackage{amsmath}
\usepackage{mathtools}
\usepackage{amsfonts}
\usepackage{comment}
\usepackage{centernot}
\usepackage{graphicx}
\usepackage{bm,dcolumn}
\usepackage{indentfirst}
\usepackage{tikz}
\usepackage{easylist}
\usepackage{enumitem}
\usetikzlibrary{arrows,calc}
\usetikzlibrary{decorations.markings}
\usepackage{easylist}
\usepackage{enumitem}
\tikzset{
    >=stealth',
    help lines/.style={thick},
    axis/.style={<-, thick},
    important line/.style={thick},
    connection/.style={thick, dotted},
    shadow lines/.style={dashed, thick},
    extended line/.style={->},
    extended line/.default=1cm]
}

\usepackage{mathtools}
\usepackage{graphicx}
\usepackage{pgfplots}

\newtheorem{theorem}{Theorem}

\newtheorem{lemma}[subsection]{{\bf Lemma}}

\newtheorem{definition}[subsection]{Definition}

\begin{document}
    \title{Linear Congruences in several variables with congruence restrictions} 
\author[Babu]{C. G. Karthick Babu}
\email{cgkarthick24@gmail.com}
\author[Bera]{Ranjan Bera}
\email{ranjan.math.rb@gmail.com}
\author[Sury]{B. Sury}
\email{surybang@gmail.com}
\address{Statistics and Mathematics Unit, Indian Statistical Institute, R.V. College Post, Bangalore-560059, India.}


\begin{abstract}
In this article, we consider systems of linear congruences in
several variables and obtain necessary and sufficient conditions as
well as explicit expressions for the number of solutions subject to
certain restriction conditions. These results are in terms of
Ramanujan sums and generalize the results of Lehmer \cite{DNL13} and
Bibak et al. \cite{BBVRL17}. These results have analogues over
$\mathbb{F}_q[t]$ where the proofs are similar, once notions such as
Ramanujan sums are defined in this set-up. We use the recent
description of Ramanujan sums over function fields as developed by
Zhiyong Zheng \cite{ZZ18}. This is discussed in the last section. We illustrate
the formulae obtained for the number of solutions through some
examples. Over the integers, such problems have a rich history, some
of which seem to have been forgotten - a number of papers written
on the topic re-prove known results. The present authors also became
aware of some of these old articles only while writing the present
article and hence, we recall very briefly some of the old work by H.
J. S. Smith, Rademacher, Brauer, Butson and Stewart, Ramanathan, McCarthy, and Spilker \cite{AB26, BS55, PJM76, HR25, KGR44, HJSS61, JS96}.

\end{abstract}

\maketitle
\pagenumbering{arabic}
\pagestyle{myheadings}
\markright{Multi variable CRT}
\section{{\bf Introduction}}

The Chinese remainder theorem (CRT) originated in the work of
Sun-Tsu in the 3rd century AD, addresses a system of congruences of
the form
\begin{align}\label{1v congs eq}
    x &\equiv b_{i} \pmod{m_{i}}, (1 \leq i \leq k).
\end{align}
If $m_{1}, \dots, m_{k}$ are pairwise coprime positive integers, and
$b_{1}, \dots, b_{k}$ are arbitrary integers, then it asserts that
the system \eqref{1v congs eq}  has a unique solution modulo $m =
m_{1} \cdots m_{k}$. Throughout the paper, we shall use the
notations $(u_{1}, \dots, u_{k})$ and $[u_{1}, \dots, u_{k}]$  to
denote the gcd and lcm of integers $u_{1}, \dots, u_{k}$
respectively.
 In 1952, Oystein Ore \cite{OO52} proved a
version for non-coprime moduli in the American Mathematical Monthly.
He showed that if the system of congruences \eqref{1v congs eq} has
solution if and only if for all $1 \leq i \neq j \leq k$,
\begin{equation}\label{gcd cond}
b_{i} \equiv b_{j} \pmod{d_{ij}}, \ \text{where} \ d_{ij}=
(m_{i},m_{j}).
\end{equation}
When the conditions \eqref{gcd cond} are satisfied, the solution of
\eqref{1v congs eq} is uniquely determined modulo the least common
multiple $m=[m_{1}, \dots, m_{k}]$.

\vspace{2mm}
\noindent
Let $a_{1}, \dots, a_{n}, b, m \in \mathbb{Z}$, $m \geq 1$.
A linear congruence in $n$ unknowns $x_{1}, \dots, x_{n}$ is of the form
\begin{equation}\label{linear cong}
a_{1}x_{1}+\dots+a_{n}x_{n} \equiv b \pmod{m}.
\end{equation}
D. N. Lehmer \cite{DNL13} showed that a linear congruence
represented by \eqref{linear cong} has a solution $(x_{1},
\dots, x_{n}) \in \mathbb{Z}_{m}^{n}$ if and only if $\ell \mid b$,
where $\ell=(a_{1}, \dots, a_{n}, m)$ and if this condition is
satisfied, then there are $\ell m^{n-1}$ many solutions. Instead of
a single linear congruence, one could consider a system of linear
congruences and explore its solutions, which can be viewed as a
multi-variable extension of the Chinese remainder theorem. This is
relatively straightforward, and was done in \cite{BS15} - but the author was unaware of Lehmer's result. More precisely, the following result was established.\\
Suppose $a_{ij}$ are integers (for $1 \leq i \leq k$, $1 \leq j \leq
n$). For positive integers $m_{1}, \dots, m_{k}$ and arbitrary
integers $b_{1}, \dots, b_{k}$, consider the system of $k$
congruences in $n$ unknowns $x_{1}, \dots, x_{n}$:
\begin{align}\label{congs eq}
    a_{11}x_{1}+a_{12}x_{2}+\dots+a_{1n}x_{n} &\equiv b_{1} \pmod{m_{1}},\nonumber\\
    a_{21}x_{1}+a_{22}x_{2}+\dots+a_{2n}x_{n} &\equiv b_{2} \pmod{m_{2}},\nonumber\\
    \cdots \cdots \cdots \cdots & \cdots \cdots\nonumber\\
    a_{k1}x_{1}+a_{k2}x_{2}+\dots+a_{kn}x_{n} &\equiv b_{k} \pmod{m_{k}}.
\end{align}

\begin{theorem}
Let $m_{1}, \dots, m_{k}$ be pairwise coprime integers. The system
of congruences above has a solution $x_{1}, \dots, x_{n}$ in
integers if and only if,  $(m_i, a_{i1}, a_{i2}, \dots,
a_{in})|b_{i}$ for each $i \leq k$.
\end{theorem}

\vspace{2mm}

\noindent In 1861, H. J. S. Smith \cite{HJSS61} first studied and
provided a necessary and sufficient condition for the existence of a solution to the system of congruences represented by
\eqref{congs eq}. If this condition is satisfied, then he also gave
the exact number of solutions of \eqref{congs eq} that are there in
$\mathbb{Z}_{m}^{n}$, where $m=[m_{1}, \dots, m_{k}]$. However,
despite the elegant matrix theoretic methods of Smith, his paper
seems to have gone unnoticed by later authors working on these
problems.

\vspace{2mm} \noindent In 1954, A. T. Butson and B. M. Stewart
\cite{BS55} presented a concise proof of Smith's result using the
concepts of invariant factors and the Smith normal form of a matrix
with elements in a principal ideal ring. Furthermore, Butson and
Stewart extended the study of systems of linear congruences modulo
an ideal and systems over a set of integral elements within an
associative algebra.

\vspace{2mm} \noindent It is worth noting that both the necessary
and sufficient condition for the existence of a solution and the
calculation of the number of solutions provided by Butson and
Stewart rely on the use of invariant factors of the Smith normal
form (see Section 3 and Section 5 of \cite{BS55}). However, calculating these factors can be challenging, as it
involves determining the greatest common divisor of determinants of
minors derived from the given matrix.

\vspace{2mm} \noindent In the first part of this paper, we revisit
and reprove the result of Lehmer \cite{DNL13} and its
straightforward generalization in \cite{BS15} using Ramanujan sums.
Additionally, we obtain expressions for the number of solutions to
the systems of linear congruences represented by
\eqref{congs eq} in terms of the greatest common divisor of the
coefficients, rather than relying on the invariant factors of the
Smith normal form. More precisely, we prove the following theorem.

\begin{theorem}\label{soln of congs thm}
Let $m_{1}, \dots, m_{k}$ be pairwise co-prime integers and
$m=m_{1}\cdots m_{k}$. The system of congruences represented by \eqref{congs eq} has a solution $\langle x_{1}, \dots,
x_{n} \rangle \in \mathbb{Z}_{m}^{n}$ if and only if for each $i
\leq k$, $\ell_{i} \mid b_{i}$, where $\ell_{i}= (a_{i1}, \dots,
a_{in}, m_{i})$, and if this condition is satisfied, then there are
$m^{n-1} \prod_{i=1}^{k} \ell_{i}$ solutions.

We note that when $k=1$, this result is the same as Lehmer's theorem
\cite{DNL13}.
\end{theorem}

\vspace{2mm} In another direction, the linear congruence represented
by \eqref{linear cong} with some restrictions on solutions
$x_i$ has also been studied in the literature. One such type of
restriction is to demand that the solutions should satisfy $(x_{i},
m)=t_{i}$ for $1 \leq i \leq n$, where $t_{1}, \dots, t_{n}$ are
given positive divisors of $m$. A linear congruence with such
restrictions is called a restricted linear congruence. Assume
$a_{1}, \dots, a_{n}, b$ are fixed and let $N_{m}(b; t_{1}, \dots,
t_{n})$ denote the number of incongruent solutions of
\begin{equation}\label{restr lin cong}
a_{1}x_{1}+\dots+a_{n}x_{n} \equiv b \pmod{m}, \ \ (x_{i}, m)=t_{i} \ \text{for } 1 \leq i \leq n.
\end{equation}
Rademacher \cite{HR25} and Brauer \cite{AB26} independently gave a
formula for the number of solutions $N_{m}(b; t_{1}, \dots, t_{n})$
of the linear congruence \eqref{restr lin cong} with $a_{i}=1$ and
$t_{i}=1$ for $1 \leq i \leq n$. Later Nicol and Vandiver
\cite{NV54} and E. Cohen \cite{CE55} obtained the following
equivalent formula:
 $$N_{m}(b; 1, \dots, 1)=\frac{1}{m}\sum_{d \mid m}C_{d}(b)(C_{m}(m/d))^{n},$$
where $C_{r}(a)$ denotes a Ramanujan sum. Also, see Ramanathan \cite{KGR44} and Spilker \cite{JS96} for further results with these and other restrictions on linear congruences. Moreover, in 1976, restricted solutions of some systems of linear congruences were studied by McCarthy \cite{PJM76} in another context.

\vspace{2mm} \noindent The study of restricted congruences has
garnered significant interest due to its applications in various
fields, such as number theory, cryptography, combinatorics, and
computer science. In \cite{VL10}, Liskovets introduced a
multivariate arithmetic function, and it can be seen that a special
case of the restricted congruence problem with $b=0$ and $a_i=1$ is
closely related to this multivariate function. This function finds
applications in both combinatorics and topology. In fact, for our purpose, we need to study this multivariate function twisted with an additive character (see \eqref{Ebm defn}). Furthermore, in
computer science, the restricted congruence problem plays a role in
the study of universal hashing, as discussed by Bibak et al.
\cite{BKVT18}.

\vspace{2mm} \noindent In 2017, Bibak et al. \cite{BBVRL17},
provided a general formula for the number of solutions $N_{m}(b;
t_{1}, \dots, t_{n})$ of the linear congruence \eqref{restr lin
cong}. They proved this using Ramanujan sums and the discrete
Fourier transform (DFT) of arithmetic functions.

In the second part of our paper, we focus on the multi-variable
analogue of their theorem. In particular, when $(x_j, m_i) =
t_{ij}$, we provide a formula for determining the number of
solutions for the aforementioned system of linear congruences. In our
approach also, we primarily use properties of Ramanujan sums, as
well as the discrete Fourier transform of arithmetic functions. The
number of solutions with $k=1$ and all the coefficients being equal
to $1$ was initially studied by Rademacher in 1925 and Brauer in
1926. Since then, the problem has been widely explored and found
significant applications in diverse areas as mentioned above. \vskip
2mm

More precisely, we establish the following result:
\begin{theorem}\label{Theorem 3}
Let $m_{1}, \dots, m_{k}$ be pairwise coprime integers and
$m=m_{1}\cdots m_{k}$. Let $t_{ij} \mid m_{i}$ (for $1 \leq i \leq
k$, $1 \leq j \leq n$). Then the number of solutions of the system
of congruences represented by \eqref{congs eq} with the
restrictions $(x_{j}, m_{i})=t_{ij}$ for $1 \leq i \leq k$, $1 \leq
i \leq n$, is given by
\[
\frac{1}{m}\prod_{j=1}^{n} \frac{\varphi(\frac{m}{t_{j}})}{\varphi(\frac{m}{t_{j}d_{j}})}
\sum_{d \mid m} C_{d}(b) \prod_{l=1}^{n}C_{\frac{m}{t_{l}d_l}}\bigg(\frac{m}{d}\bigg).
\]
where $t_j= \prod_{i=1}^{k}t_{ij}$
and $d_j= \prod_{i=1}^{k}d_{ij}$ with $d_{ij}=(a_{ij}, \frac{m_i}{t_{ij}})$
for $1\le i\le k$, $1\le j\le n$.
Here $b$ represents the unique
solution of the following system of congruences

\begin{align}\label{uniq b}
x &\equiv b_{i} \pmod{m_{i}},1 \leq i \leq k.
\end{align}
\end{theorem}

\section{Ramanujan Sums and DFT}
Let $e(x)$ denote $e^{2\pi ix}$. For integers $a$ and $m \geq 1$, the Ramanujan sum $C_{m}(a)$ is defined by
\begin{equation*}
    C_{m}(a)= \sum_{\substack{j=1 \\ (j, m)=1}}^{m}e\bigg(\frac{ja}{m}\bigg).
\end{equation*}

\noindent
It is well known that $C_{m}(a)$ is integer-valued and even function of $a$, that is, $C_{m}(a)=C_{m}((a,m))$, for every $a, m$. Moreover,  $C_{m}(a)$ has the following explicit formula:
\begin{equation}\label{explicit formula for RS}
 C_{m}(a)= \sum_{d \mid (a,m)} \mu\bigg(\frac{m}{d}\bigg) d,
\end{equation}

Recall that an arithmetic function is a function $f: \mathbb{N} \rightarrow \mathbb{C}$. More generally, an arithmetic function of $n$ variables is a function  $f: \mathbb{N}^{n} \rightarrow \mathbb{C}$. Let $\mathcal{F}_{n}$ denotes the set of all arithmetic function of $n$ variables. If $f, g \in \mathcal{F}_{n}$, then their convolution is defined as
\begin{equation*}
(f\ast g)(m_{1}, \dots, m_{n})= \sum_{d_{1} \mid m_{1}, \dots, d_{n} \mid m_{n}} f(d_{1}, \dots, d_{n}) g\big(\frac{m_{1}}{d_{1}}, \dots, \frac{m_{n}}{d_{n}}\big).
\end{equation*}

The set $\mathcal{F}_{n}$ forms a ring with pointwise addition and convolution product with unit element $\varepsilon^{(n)}$ defined by
\begin{equation*}
\varepsilon^{(n)}(m_{1}, \dots, m_{n})=\begin{cases}
    1 ,& \ \ \text{if } \  m_{1}= \dots = m_{n}=1,\\
     0, & \ \ \text{otherwise}.
    \end{cases}
\end{equation*}
A function $f \in \mathcal{F}_{n}$ is invertible iff $f(1, \dots, 1) \neq 0$. In fact, we have the M\"{o}bius inversion formula:
\begin{lemma}\label{MIF lemma}
If $f,g \in \mathcal{F}_{n}$ satisfying
\begin{equation*}
f(m_{1}, \dots, m_{n})=\sum_{d_{1} \mid m_{1}, \dots, d_{n} \mid m_{n}} g(d_{1}, \dots, d_{n})
\end{equation*}
then
\begin{equation*}
g(m_{1}, \dots, m_{n})=\sum_{d_{1} \mid m_{1}, \dots, d_{n} \mid m_{n}} f(d_{1}, \dots, d_{n}) \mu(m_{1}/d_{1}) \cdots \mu(m_{n}/d_{n}).
\end{equation*}
Equivalently, the inverse of the constant function $1$ is given by $$\mu^{(n)}(m_{1}, \dots, m_{n})=\mu(m_{1}) \cdots \mu(m_{n}).$$
\end{lemma}
 An arithmetic function $f$ is said to be periodic with period $r$ (or $r$-periodic) for some $r \in \mathbb{N}$ if for every $m \in \mathbb{Z}$, $f(m+r)=f(m)$.
 For an $r$-periodic arithmetic function $f$, its discrete (finite)
 Fourier transform (DFT) is defined to be the function
\begin{equation*}
\hat{f}(b)= \sum_{j=1}^{m} f(j) e\bigg(\frac{-bj}{m}\bigg), \  \ \text{for} \
 \ b \in \mathbb{Z}.
\end{equation*}

\noindent
A Fourier representation of $f$ is given by
\begin{equation*}
f(b)= \frac{1}{m}\sum_{j=1}^{m} \hat{f}(j) e\bigg(\frac{bj}{m}\bigg), \  \ \text{for} \
 \ b \in \mathbb{Z}.
\end{equation*}

Recall that an arithmetic function $f$ is $r$-even if $f(m)=f((m,r))$, for every $m \in \mathbb{Z}$. Clearly, if a function $f$ is $r$-even, then it is $r$-periodic. Further, for an $r$-even function $f$, we have
\begin{align*}
\hat{f}(b)=& \sum_{k=1}^{r} f(k) e\bigg(\frac{-bk}{r}\bigg)= \sum_{d \mid r} \sum_{\substack{1 \leq j \leq \frac{r}{d} \\ (j, \frac{r}{d})=1}}f(d) e\bigg(\frac{-bdj}{r}\bigg)= \sum_{d \mid r} f(d) C_{\frac{r}{d}}(b).
\end{align*}

The Cauchy convolution of two $r$-periodic functions $f$ and $g$ is
defined to be
$$(f \otimes g)(n)= \sum_{\substack{1 \leq x, y \leq m \\ x+y \equiv n \pmod{m}}} f(x) g(y)=\sum_{x=1}^{m} f(x) g(n-x), \ \text{for} \ m \in \mathbb{Z}.$$
Similarly, we can define the Cauchy convolution of a finite number of $r$-periodic functions. It is easy to observe that the discrete Fourier transform of the Cauchy convolution satisfies the relation $$\widehat{f \otimes g} = \hat{f} \hat{g},$$ with pointwise multiplication.
\section{Preliminary Lemmas}
In this section, we list several elementary lemmas required for the
proof of theorems 2 and 3. 
\subsection{Some properties of GCD and LCM}\label{lcm and gcd sec}
For a prime $p$ and a nonnegative integer
$k$, by the notation $p^{k} \Vert a$, we mean $p^{k} \mid a$ and
$p^{k+1} \nmid a$.
\begin{lemma}\label{coprime lcmgcd lem}
Let $m, m_{1}, m_{2}$ be positive integers such that $(m_{1},m_{2})=1$. Then, we have
\begin{enumerate}[label=(\roman*)]
\item $(m_{1}m_{2}, m)=(m_{1},m)\cdot (m_{2}, m)$. \label{gcd coprime}
\item $[m_{1}m_{2}, m]=\frac{[m_{1},m]\cdot [m_{2}, m]}{m}$. \label{lcm coprime}
\end{enumerate}
\end{lemma}

\vspace{1mm}
\noindent
The equality \ref{gcd coprime} is trivial and \ref{lcm coprime} follows from the identity $[m,n]\cdot(m,n)=mn$ and \ref{gcd coprime}.

\begin{lemma}\label{1var lcmgcd lem}
Let $m, a_{1}, a_{2}, \dots, a_{n}$ be positive integers and $d_{i}=(a_{i},m)$, for $1 \le i \leq n$. Then, we have
\begin{equation*}
\bigg[\frac{m}{d_{1}}, \dots, \frac{m}{d_{n}}\bigg]= \frac{m}{(a_{1}, \dots, a_{n}, m)}.
\end{equation*}
\end{lemma}
\begin{proof}
Let $p$ be any prime dividing $m$. Suppose $p^{k} \Vert m$ and $p^{k_{1}} \Vert d_{1}, \dots, p^{k_{n}} \Vert d_{n}$, then we have
$$p^{k-\min\{k_{1}, \dots, k_{n}\}}=p^{\max\{k-k_{1}, \dots, k-k_{n}\}} \bigg\Vert \bigg[\frac{m}{d_{1}}, \dots, \frac{m}{d_{n}}\bigg].$$
On the other hand, from $p^{\min\{k_{1}, \dots, k_{n}\}} \Vert (a_{1}, \dots, a_{n}, m)$, it follows that
$$p^{k-\min\{k_{1}, \dots, k_{n}\}} \bigg\Vert \frac{m}{(a_{1}, \dots, a_{n}, m)}.$$
\end{proof}

\begin{lemma}\label{lcmgcd lem}
Let $k$ and $n$ be arbitrary positive integers and $a_{ij}$ are integers (for $1 \leq i \leq k$, $1 \leq j \leq n$). Let $m_{1}, \dots, m_{k}$ be pairwise coprime integers and $d_{ij}=(a_{ij}, m_{i})$ (for $1 \leq i \leq k$, $1 \leq j \leq n$). Then, we have
\begin{equation*}
\bigg[\frac{m_{1} \cdots m_{k}}{d_{11}\cdots d_{k1}}, \dots, \frac{m_{1} \cdots m_{k}}{d_{1n}\cdots d_{kn}}\bigg]= \frac{m_{1} \cdots m_{k}}{\ell_{1} \cdots \ell_{k}},
\end{equation*}
where $\ell_{i}=(a_{i1}, \dots, a_{in}, m_{i}),$ for $1 \leq i \leq k$.
\end{lemma}
\begin{proof}
We prove this lemma by induction on $n$. For $n=1$, the claim is trivial. Now, we prove the claim for $n=2$. Since $m_{1}, \dots, m_{k}$ are pairwise coprime integers, by using Lemma \ref{coprime lcmgcd lem}-\ref{lcm coprime} $(k-1)$ times successively, we write
\begin{align*}
\bigg[\frac{m_{1} \cdots m_{k}}{d_{11}\cdots d_{k1}}, \frac{m_{1} \cdots m_{k}}{d_{12}\cdots d_{k2}}\bigg]=&\bigg[\frac{m_{1}}{d_{11}}, \frac{m_{1} \cdots m_{k}}{d_{12}\cdots d_{k2}}\bigg] \cdots \bigg[\frac{m_{k}}{d_{k1}}, \frac{m_{1} \cdots m_{k}}{d_{12}\cdots d_{k2}}\bigg]\times \bigg(\frac{d_{12}\cdots d_{k2}}{m_{1} \cdots m_{k}}\bigg)^{k-1}\\
=&\bigg[\frac{m_{1}}{d_{11}}, \frac{m_{1}}{d_{12}}\bigg] \cdots \bigg[\frac{m_{k}}{d_{k1}}, \frac{m_{k}}{d_{k2}}\bigg]\times \bigg(\frac{m_{1} \cdots m_{k}}{d_{12}\cdots d_{k2}}\bigg)^{k-1}\times \bigg(\frac{d_{12}\cdots d_{k2}}{m_{1} \cdots m_{k}}\bigg)^{k-1}\\
=&\bigg[\frac{m_{1}}{d_{11}}, \frac{m_{1}}{d_{12}}\bigg] \cdots \bigg[\frac{m_{k}}{d_{k1}}, \frac{m_{k}}{d_{k2}}\bigg].
\end{align*}
Thus, the required equality follows from Lemma \ref{1var lcmgcd lem}. This proves the claim for $n=2$. Now, assuming that the lemma holds for $r \leq n-1$, we can use the associativity of LCM and the induction hypothesis to deduce:
\begin{align*}
\bigg[\frac{m_{1} \cdots m_{k}}{d_{11}\cdots d_{k1}}, \dots, \frac{m_{1} \cdots m_{k}}{d_{1n}\cdots d_{kn}}\bigg]=&\bigg[\bigg[\frac{m_{1} \cdots m_{k}}{d_{11}\cdots d_{k1}}, \dots, \frac{m_{1} \cdots m_{k}}{d_{1(n-1)}\cdots d_{k(n-1)}}\bigg], \frac{m_{1} \cdots m_{k}}{d_{1n}\cdots d_{kn}}\bigg]\\
=&\bigg[\frac{m_{1} \cdots m_{k}}{\ell'_{1}\cdots \ell'_{k}}, \frac{m_{1} \cdots m_{k}}{d_{1n}\cdots d_{kn}}\bigg],
\end{align*}
where $\ell'_{i}=(a_{i1}, \dots, a_{i(n-1)}, m_{i}),$ for $1 \leq i \leq k$. Since $m_{1}, \dots, m_{k}$ are pairwise coprime integers, by following a similar argument as in the case of $n=2$, we obtain
$$\bigg[\frac{m_{1} \cdots m_{k}}{\ell'_{1}\cdots \ell'_{k}}, \frac{m_{1} \cdots m_{k}}{d_{1n}\cdots d_{kn}}\bigg]=\bigg[\frac{m_{1}}{\ell'_{1}}, \frac{m_{1}}{d_{1n}}\bigg] \cdots \bigg[\frac{m_{k}}{\ell'_{k}}, \frac{m_{k}}{d_{kn}}\bigg].$$
It follows from Lemma \ref{1var lcmgcd lem} that
$$\bigg[\frac{m_{1}}{\ell'_{1}}, \frac{m_{1}}{d_{1n}}\bigg] \cdots \bigg[\frac{m_{k}}{\ell'_{k}}, \frac{m_{k}}{d_{kn}}\bigg]=\frac{m_{1} \cdots m_{k}}{\ell_{1} \cdots \ell_{k}},$$
where $\ell_{i}=(a_{i1}, \dots, a_{in}, m_{i}),$ for $1 \leq i \leq k$. This completes the proof of Lemma \ref{lcmgcd lem}.
\end{proof}

\vspace{1mm}
\noindent
\subsection{Some applications of Ramanujan Sums and DFT}
For $t \mid r$, let $\varrho_{r,t}$ be the $r$-periodic function defined for every $k \in \mathbb{Z}$ by
\begin{equation*}
    \varrho_{r,t}(k)= \begin{cases}
    1 ,& \text{if } (k, r)=t ,\\
     0, & \text{if } (k, r) \neq t.
    \end{cases}
\end{equation*}
It follows from (Theorem 2.5, \cite{BBVRL17}) that the Ramanujan sum $k \rightarrow C_{r}(k)$ is the DFT of the function $k \rightarrow \varrho_{r,1}(k)$. More generally, $$\widehat{\varrho_{r,t}}(k)=C_{\frac{r}{t}}(k) \ \text{for } k \in \mathbb{Z}.$$ Recall that $N_{m}(b; t_{1}, \dots, t_{n})$ denotes the number of incongruent solutions of the restricted linear congruence \eqref{restr lin cong}.
The following lemma is a consequence of the DFT of $\varrho_{r,t}$ and Cauchy convolution, which provides a formula for $N_{m}(b; t_{1}, \dots, t_{n})$ in terms of Ramanujan's sum when $a_{i}=1$, for $1 \leq i \leq n$.
\begin{lemma}(Theorem 3.4, \cite{BBVRL17})\label{without coeff lem}
Let $b,m \geq 1$, $t_{i} \mid m$ for $1 \leq i \leq n$ be given integers. The number of solutions of the linear congruence $x_{1}+ \dots+x_{n} \equiv b \pmod m$, with $(x_{i}, m)=t_{i}$ for $1 \leq i \leq n$, is
\begin{equation*}
N_{m}(b; t_{1}, \dots, t_{n})= \frac{1}{m} \sum_{j=1}^{m} C_{\frac{m}{t_{1}}}(j)\cdots C_{\frac{m}{t_{n}}}(j) e\bigg(\frac{bj}{m}\bigg) =\frac{1}{m} \sum_{d \mid m} C_{d}(b) \prod_{i=1}^{n}C_{\frac{m}{t_{i}}}\bigg(\frac{m}{d}\bigg) \geq 0.
\end{equation*}
\end{lemma}

\begin{lemma}(Theorem 3.1, \cite{BBVRL17})\label{lemma 6}
Let $a,b, m \geq 1$ and $t \geq 1$ be the given integers. The congruence $ax \equiv b
\pmod{m}$ has a solution $x$ with $(x, m)=t$ if and only if $t \mid (b,m)$ and $\big(a,
\tfrac{m}{t}\big)=\big(\tfrac{b}{t}, \tfrac{m}{t}\big)$. Furthermore, if these
conditions are satisfied, then there are exactly
\begin{equation*}
\frac{\varphi(\tfrac{m}{t})}{\varphi{(\tfrac{m}{td})}}=d \prod_{\substack{p \mid d \\ p \nmid \frac{m}{td}}} \bigg(1- \frac{1}{p}\bigg)
\end{equation*}
solutions, where $p$ ranges over the primes and $d=\big(a, \tfrac{m}{t}\big)=\big(\tfrac{b}{t}, \tfrac{m}{t}\big)$.
\end{lemma}

\vspace{2mm} \noindent 
For any $m_{1}, \dots, m_{n}, b \in \mathbb{N}$ and for any $m \in \mathbb{N}$ such
 that $[m_{1}, \dots, m_{n}] \mid m$, we define
\begin{equation}\label{Ebm defn}
E(b; m_{1}, \dots, m_{n}): = \frac{1}{m} \sum_{j=1}^{m} C_{m_{1}}(j) \cdots C_{m_{n}}(j)e\bigg(\frac{bj}{m}\bigg).
\end{equation}
We would like to mention that the special case of $b=0$ gives the orbicyclic (multivariate arithmetic) function  
$$E(m_{1}, \dots, m_{n})=\frac{1}{m} \sum_{j=1}^{m} C_{m_{1}}(j) \cdots C_{m_{n}}(j),$$ 
established in \cite{VL10}. The orbicyclic function, $E(m_{1}, \dots, m_{n})$, has very interesting combinatorial and topological applications, particularly in counting non-isomorphic mappings on orientable surfaces, and was studied in \cite{VL10, PJM76, MN06}. 
\begin{lemma}\label{MI of E}
For any $m_{1}, \dots, m_{n}, b \in \mathbb{N}$, we have
\begin{equation*}
\sum_{d_{1} \mid m_{1}, \dots, d_{n} \mid m_{n}} E(b;d_{1}, \dots, d_{n})= J(b; m_{1}, \dots, m_{n}),
\end{equation*}
where
\begin{equation*}
J(b; m_{1}, \dots, m_{n})= \begin{cases}
    \frac{m_{1} \cdots m_{n}}{[m_{1}, \dots, m_{n}]} ,& \text{if }   \frac{m}{[m_{1}, \dots, m_{n}]} \mid b\\
     0, & \text{otherwise}.
    \end{cases}
\end{equation*}
\end{lemma}
\begin{proof}
By using \eqref{explicit formula for RS} in \eqref{Ebm defn}, we write
\begin{align*}
E(b; m_{1}, \dots, m_{n}) =& \frac{1}{m} \sum_{j=1}^{m}e\bigg(\frac{bj}{m}\bigg) \sum_{d_{1} \mid (j,m_{1})} d_{1} \mu(m_{1}/ d_{1}) \dots  \sum_{d_{n} \mid (j,m_{n})} d_{n} \mu(m_{n}/ d_{n})\\
=& \frac{1}{m} \sum_{d_{1} \mid m_{1}, \dots, d_{n} \mid m_{n}} d_{1} \mu(m_{1}/ d_{1}) \cdots d_{n} \mu(m_{n}/ d_{n}) \sum_{\substack{1 \leq j \leq m \\ d_{1} \mid j, \dots, d_{n} \mid j}}e\bigg(\frac{bj}{m}\bigg)\\
=& \frac{1}{m} \sum_{d_{1} \mid m_{1}, \dots, d_{n} \mid m_{n}} d_{1} \dots d_{n} \ \mu(m_{1}/ d_{1}) \cdots \mu(m_{n}/ d_{n}) \sum_{\substack{1 \leq j \leq m \\ [d_{1}, \dots, d_{n}] \mid j}}e\bigg(\frac{bj}{m}\bigg)\\
=& \sum_{d_{1} \mid m_{1}, \dots, d_{n} \mid m_{n}} J(b; d_{1}, \dots, d_{n}) \mu(m_{1}/ d_{1}) \cdots  \mu(m_{n}/ d_{n}).
\end{align*}
Now using Lemma \ref{MIF lemma} (M\"{o}bius inversion formula), we obtain
\begin{equation*}
\sum_{d_{1} \mid m_{1}, \dots, d_{n} \mid m_{n}} E(b;d_{1}, \dots, d_{n})= J(b; m_{1}, \dots, m_{n}).
\end{equation*}
This completes the proof of Lemma \ref{MI of E}.
\end{proof}

\section{Solutions of Linear Congruences}\label{soln of lin cong}
In this section, we present the proof of Theorem \ref{soln of congs thm} and Theorem \ref{Theorem 3}. Furthermore, we compare the solution count given in Theorem \ref{soln of congs thm} with the solution count obtained by Butson and Stewart \cite{BS55} through the examination of specific examples. Also, we discuss an example for Theorem \ref{Theorem 3} and verify the solution count by the first principle.

\subsection{Proof of Theorem \ref{soln of congs thm}. }First, we consider the system of  congruences
\begin{align*}
    a_{1j}x_{j} &\equiv y_{j} \pmod{m_{1}},\nonumber\\
    a_{2j}x_{j} &\equiv y_{j} \pmod{m_{2}},\nonumber\\
    &\cdots \cdots \cdots \cdots \nonumber\\
    a_{kj}x_{j} &\equiv y_{j} \pmod{m_{k}}, \ \ \ \text{for} \ 1 \leq j \leq n.
\end{align*}
Since $m_{1}, \dots, m_{k}$ are pairwise coprime integers, it is clear that the system of congruences has a solution if and only if $d_{ij}=(a_{ij}, m_{i}) \mid y_{j}$ for $1 \leq i \leq k$. If this condition is satisfied, then there are $d_{j}=\prod_{i=1}^{k}d_{ij}$ solutions for each $1 \leq j \leq n$. Therefore, the problem reduces to counting the number of solutions of the system of restricted congruences
 \begin{align*}
    y_{1}+y_{2}+\dots+y_{n} &\equiv b_{1} \pmod{m_{1}},\nonumber\\
     y_{1}+y_{2}+\dots+y_{n} &\equiv b_{2} \pmod{m_{2}},\nonumber\\
    \cdots \cdots \cdots \cdots & \cdots \cdots\nonumber\\
     y_{1}+y_{2}+\dots+y_{n} &\equiv b_{k} \pmod{m_{k}}.
\end{align*}
with $(m_{i}, y_{j})=t_{ij}$ $(1 \leq i \leq k$, $1 \leq j \leq n)$, for each $d_{ij} \mid t_{ij} \mid m_{i}, \ (1 \leq i \leq k, 1 \leq j \leq n)$. Given that $m_{1}, \dots, m_{k}$ are pairwise coprime integers, the aforementioned restrictions can be expressed as $(y_{j}, m)=t_{j}$ for $1 \leq j \leq n$, where $d_{j} \mid t_{j} \mid m$ for $1 \leq j \leq n$ and $m=m_{1} \cdots m_{k}$. By applying the one-variable Chinese remainder theorem, we can reformulate the problem as counting the solutions to the restricted linear congruence
$$y_{1}+\dots+y_{n} \equiv b \pmod{m},$$
subject to the conditions $(y_{j}, m)=t_{j}$ for $1 \leq j \leq n$. Here, $b < m$ represents the unique solution to the system of congruences
\begin{align}\label{k-cong to 1cong eq}
x &\equiv b_{1} \pmod{m_{1}},\nonumber\\
x &\equiv b_{2} \pmod{m_{2}},\nonumber\\
&\cdots \cdots \cdots \cdots \nonumber \\
x &\equiv b_{k} \pmod{m_{k}}.
\end{align}
Thus, by using Lemma \ref{without coeff lem}, the number of solutions of \eqref{congs eq} is given by
  \begin{align*}
N_{m}(b; n)&= d_{1} \cdots d_{n} \sum_{d_{1} \mid t_{1} \mid m, \dots, d_{n} \mid t_{n} \mid m} N_{m}(b; t_{1}, \dots, t_{n})\\
&= d_{1} \cdots d_{n} \sum_{d_{1} \mid t_{1} \mid m, \dots, d_{n} \mid t_{n} \mid m} \frac{1}{m} \sum_{j=1}^{m} C_{\frac{m}{t_{1}}}(j)\cdots C_{\frac{m}{t_{n}}}(j) e\bigg(\frac{bj}{m}\bigg)\\
&=d_{1} \cdots d_{n}\sum_{t_{1} \mid \frac{m}{d_{1}}, \dots, t_{n} \mid \frac{m}{d_{n}}} \frac{1}{m} \sum_{j=1}^{m} C_{\frac{m}{t_{1}d_{1}}}(j)\cdots C_{\frac{m}{t_{n}d_{n}}}(j) e\bigg(\frac{bj}{m}\bigg)\\
&=d_{1} \cdots d_{n} \sum_{t_{1} \mid \frac{m}{d_{1}}, \dots, t_{n} \mid \frac{m}{d_{n}}}  E(b; t_{1}, \dots, t_{n}).
\end{align*}
Then, it follows from Lemma \ref{MI of E} that
$$N_{m}(b; n)=d_{1} \cdots d_{n} J\bigg(b; \frac{m}{d_{1}}, \dots, \frac{m}{d_{n}}\bigg)=\begin{cases}
    \frac{m^{n}}{[\frac{m}{d_{1}}, \dots, \frac{m}{d_{n}}]} ,& \text{if }   \frac{m}{[\frac{m}{d_{1}}, \dots, \frac{m}{d_{n}}]} \mid b\\
     0, & \text{otherwise}.
    \end{cases}$$
Recall that $[\frac{m}{d_{1}}, \dots, \frac{m}{d_{n}}]=\bigg[\frac{m_{1} \cdots m_{k}}{d_{11}\cdots d_{k1}}, \dots, \frac{m_{1} \cdots m_{k}}{d_{1n}\cdots d_{kn}}\bigg]$, where $d_{ij}=(a_{ij}, m_{i})$ (for $1 \leq i \leq k$, $1 \leq j \leq n$). Therefore, By using Lemma \ref{lcmgcd lem}, we have
$$N_{m}(b; n)=\begin{cases}
    m^{n-1} \prod_{i=1}^{k} \ell_{i}, & \text{if } \prod_{i=1}^{k} \ell_{i} \mid b\\
     0, & \text{otherwise},
    \end{cases}$$
where $\ell_{i}=(a_{i1}, \dots, a_{in}, m_{i}),$ for $1 \leq i \leq k$. Since $b$ satisfies the system of congruences represented by \eqref{k-cong to 1cong eq}, we can express the above condition as follows.
$$N_{m}(b; n)=\begin{cases}
    m^{n-1} \prod_{i=1}^{k} \ell_{i}, & \text{if }  \ell_{i} \mid b_{i} \ \text{for every } 1 \leq i \leq k\\
     0, & \text{otherwise}.
    \end{cases}$$
This completes the proof of Theorem \ref{soln of congs thm}.

\subsection{Some examples} (i). Consider the system of congruences:
\begin{align*}
2 x_{1}+2x_{2} &\equiv 2 \pmod{12}\\
5x_{1}+7x_{2} &\equiv 1 \pmod{35}.
\end{align*}

Using Theorem \ref{soln of congs thm}, we find that this system has a total of $840$ solutions. To calculate the solution count using Butson and Stewart's method, we first simplify the congruences by multiplying suitable factors:
\begin{align*}
70 x_{1}+70 x_{2} &\equiv 70 \pmod{420}\\
60 x_{1}+84 x_{2} &\equiv 12 \pmod{420}.
\end{align*}

The number of solutions can be determined by computing the invariant factors of the Smith normal form of the matrix:
\begin{equation*}
A =
\begin{pmatrix}
70 & 70 \\
60 & 84
\end{pmatrix}.
\end{equation*}

Calculating the invariant factors as $e_{1}=2$ and $e_{2}=840$, we obtain the solution count as $(2,420)\cdot (840, 420)=840$. Thus, both methods yield the same number of solutions.

\vspace{2mm}
\noindent
(ii). Consider the system of congruences:
\begin{align*}
3 x_{1}+6x_{2}+3x_{3} &\equiv 3 \pmod{9}\\
4x_{1}+2x_{2}+8x_{3} &\equiv 4 \pmod{16}\\
2x_{1}+3x_{2}+x_{3} &\equiv 2 \pmod{5}.
\end{align*}

Using Theorem \ref{soln of congs thm}, we find that this system has a total of $3110400$ solutions. To calculate the solution count using Butson and Stewart's method, we first simplify the congruences by multiplying suitable factors:
\begin{align*}
240 x_{1}+480x_{2}+240x_{3} &\equiv 240 \pmod{720}\\
180x_{1}+90x_{2}+360x_{3} &\equiv 180 \pmod{720}\\
288x_{1}+432x_{2}+144_{3} &\equiv 288 \pmod{720}.
\end{align*}
The number of solutions can be determined by computing the invariant factors of the Smith normal form of the matrix:
\begin{equation*}
A =
\begin{pmatrix}
240 & 480 & 240\\
180 & 90 & 360\\
288 & 432 & 144
\end{pmatrix}.
\end{equation*}
Computationally, we found that the invariant factors of the matrix are $e_{1}=6$, $e_{2}=720$, and $e_{3}=3600$. Therefore, the number of solutions is equal to $(6,720)\cdot (720,720) \cdot (3600,720)=3110400$. Thus, both methods yield the same number of solutions.

\subsection{Proof of Theorem \ref{Theorem 3}. }
Assume that the system of linear congruences in \eqref{congs eq} has
a solution $\langle x_{1}, \dots, x_{n} \rangle \in
\mathbb{Z}_{m}^{n}$ with $(x_{j}, m_{i})=t_{ij}$ for $1 \leq i \leq
k$, $1 \leq i \leq n$. Since $m_1,m_2,\ldots m_k$ are co-prime,  it
follows from Lemma \ref{lemma 6} that the system of congruences

\begin{align}\label{Equation 9}
a_{1j}x_{j} &\equiv y_{j} \pmod{m_{1}},\nonumber\\
a_{2j}x_{j} &\equiv y_{j} \pmod{m_{2}},\nonumber\\
&\cdots \cdots \cdots \cdots \nonumber \\
a_{kj}x_{j} &\equiv y_{j} \pmod{m_{k}}, \quad \text{for} \ 1 \leq j \leq n
\end{align} has a solution if and only if $t_{ij}=(x_{j}, m_{i}) \mid (y_{j},m_i)$ for $1 \leq i \leq k$, $1 \leq j \leq n$.
Then $(a_{ij}x_j,m_i)=(y_j,m_i)=t_{ij}d_{ij}$ for some $d_{ij}$, where $1 \leq i \leq k$, $1 \leq j \leq n$. Therefore,
\begin{align*}
    \left(\frac{a_{ij}x_j}{t_{ij}}, \frac{m_i}{t_{ij}}\right) &= \left(\frac{y_j}{t_{ij}}, \frac{m_i}{t_{ij}}\right) = d_{ij} \quad \text{for } 1 \leq i \leq k,
    1 \leq j \leq n.
\end{align*}
 As $(\frac{x_j}{t_{ij}},\frac{m_i}{t_{ij}})=1$, we obtain
 $d_{ij}=(a_{ij}, \frac{m_i}{t_{ij}})= (\frac{y_j}{t_{ij}}, \frac{m_i}{t_{ij}})$.
 Hence, the problem boils down to counting the number of solutions in the system of congruences
 \begin{align*}
    y_{1}+y_{2}+\dots+y_{n} &\equiv b_{1} \pmod{m_{1}},\nonumber\\
     y_{1}+y_{2}+\dots+y_{n} &\equiv b_{2} \pmod{m_{2}},\nonumber\\
    \cdots \cdots \cdots \cdots & \cdots \cdots\nonumber\\
     y_{1}+y_{2}+\dots+y_{n} &\equiv b_{k} \pmod{m_{k}}
\end{align*}
with $(y_j,m_i)=d_{ij}t_{ij}$. As $m_1,m_2, \ldots, m_k$ are co-prime, the above system
of congruences reduced to
$y_{1}+y_{2}+\dots+y_{n} \equiv b \pmod m$
with $(y_j,m)=t_jd_j$, where $m=m_1m_2\cdots m_k$, $t_j= \prod_{i=1}^{k}t_{ij}$, $d_j= \prod_{i=1}^{k}d_{ij}$ for $1\le j\le n$ and $b$ represents the unique solution to the system of congruences given by \eqref{uniq b}.

Therefore, by Lemma \ref{without coeff lem}, the number of solutions
of the linear
 congruence $y_{1}+y_{2}+\dots+y_{n} \equiv b \pmod m$
with $(y_j,m)=t_jd_j\, (1\le j \le n)$, is
\[
\frac{1}{m} \sum_{d \mid m} C_{d}(b) \prod_{l=1}^{n}C_{\frac{m}{t_{l}d_l}}\bigg(\frac{m}{d}\bigg)
\]
Again, for a given solution $\langle y_1,y_2,\ldots, y_n\rangle\in \mathbb{Z}$ in $y_{1}+y_{2}+\dots+y_{n} \equiv b \pmod m$
with $(y_j,m)=t_jd_j\, (1\le j \le n)$, we need to find out the number of solution in
$a_{ij}x_j\equiv y_j\pmod {m_{i}}$, with $(x_j,m_i)=t_{ij}$ for $1\le i\le k$.
Recall that $d_{ij}=(a_{ij}, \frac{m_i}{t_{ij}})= (\frac{y_j}{t_{ij}}, \frac{m_i}{t_{ij}})$, thus by Lemma \ref{lemma 6}, for fixed $j$, the above congruence has exactly
\[
\prod_{i=1}^{k}\frac{\varphi(\frac{m_i}{t_{ij}})}{\varphi{(\frac{m_i}{t_{ij}d_{ij}})}}
\]
solutions. Hence, for $1\le j\le n$ the system of congruence \eqref{Equation 9}
 has exactly
\[
\prod_{i=1}^{k}\prod_{j=1}^{n} \frac{\varphi(\frac{m_i}{t_{ij}})}{\varphi(\frac{m_i}{t_{ij}d_{ij}})}
\]
solutions. Since $m_i$'s $(1\le i\le k)$ are pairwise co-prime, the above product
can be written as
\[
\prod_{j=1}^{n} \frac{\varphi(\frac{m}{t_{j}})}{\varphi(\frac{m}{t_{j}d_{j}})}
\]
Therefore, the number of solutions of the system of congruences \eqref{congs eq}
is exactly
\[
\frac{1}{m}\prod_{j=1}^{n} \frac{\varphi(\frac{m}{t_{j}})}{\varphi(\frac{m}{t_{j}d_{j}})}  \sum_{d \mid m} C_{d}(b) \prod_{l=1}^{n}C_{\frac{m}{t_{l}d_l}}\bigg(\frac{m}{d}\bigg).
\]

\subsection{Example}
We consider the system of congruences
\begin{align}\label{eq0}
&X_1+2X_2\equiv 7 \pmod {15}\nonumber\\ \quad &3X_1+ X_2\equiv 9\pmod {14}.
\end{align}
with the restrictions
$$(X_1,15)=5, (X_1,14)=2, (X_2,15) = 3, (X_2,14)=7.$$
By using the theorem, we show that it has a unique solution for
$X_1, X_2$ modulo $210 = 15 \times 14$. We verify this by first
principles also.

\vspace{1mm}
\noindent
In the notation of the theorem, $m=210$, $t_{11}=5,
t_{12}=3, t_{21}=2, t_{22}=7$, so $t_1=10, t_2 =21$ and $d_{ij}=1$ for
$1\le i,j\le 2$, hence $d_i=1$ for $1\le i\le 2$. It is easy to see that $b=37$ is the unique solution of the congruences $b \equiv 7 \pmod
{15}, b \equiv 9 \pmod{14}$. Then, it follows from Theorem \ref{Theorem 3} that number of solutions of \eqref{eq0} is equal to
\begin{align*}
&\frac{1}{210} \sum_{d \mid 210}  C_{d}(37) \times C_{21}\big(\frac{210}{d}\big) \times C_{10}\big(\frac{210}{d}\big).
\end{align*}

\noindent
Now, we compute the product of the Ramanujan sums for each divisor $d$ of $210$ in the following table:
\begin{center}
\begin{tabular}{ | m{4em} | m{1.5cm}| m{2cm} |m{2cm} | m{2cm}| m{2cm}| m{1.5cm}| }
  \hline
  Sl. no& $d=$ & $C_{d}(37)$ & $C_{21}(\frac{210}{d})$ & $C_{10}(\frac{210}{d})$ & Product \\
  \hline
  1 & $1$ & $1$ & $12$ & $4$& $48$ \\
  \hline
  2 & $2$ & $-1$ & $12$ & $-4$& $48$ \\
  \hline
  3 & $3$ & $-1$ & $-6$ & $4$& $24$\\
  \hline
  4 & $5$ & $-1$ & $12$ & $-1$& $12$\\
  \hline
  5 & $7$ & $-1$ & $-2$ & $4$& $8$\\
  \hline
  6 & $6$ & $1$ & $-6$ & $-4$& $24$\\
  \hline
  7 & $10$ & $1$ & $12$ & $1$& $12$ \\
  \hline
  8 & $14$ & $1$ & $-2$ & $-4$& $8$ \\
  \hline
  9 & $15$ & $1$ & $-6$ & $-1$& $6$\\
  \hline
  10 & $21$ & $1$ & $1$ & $4$& $4$\\
  \hline
  11 & $35$ & $1$ & $-2$ & $-1$& $2$\\
  \hline
  12 & $30$ & $-1$ & $-6$ & $1$& $6$\\
  \hline
  13& $42$ & $-1$ & $1$ & $-4$& $4$\\
  \hline
  14& $70$ & $-1$ & $-2$ & $1$& $2$\\
  \hline
  15& $105$ & $-1$ & $1$ & $-1$& $1$\\
  \hline
  16& $210$ & $1$ & $1$ & $1$& $1$\\
  \hline
\end{tabular}
\end{center}
Hence, the number of solutions is equal to
$$\frac{1}{210}(48+48+24+12+8+24+12+8+6+4+2+6+4+2+1+1)=1.$$

\noindent \textbf{Verification by first principles: }First, note
that $(10,21)$ is a solution; call this $(x,y)$. If $(u,v)$ is any
solution of the system of congruences - without assuming the GCD
conditions, then $(x-u,y-v)$ is a solution $(s,t)$ of the system
\begin{align*}
&X_1+2X_2\equiv 0 \pmod {15}\\ \quad &3X_1+ X_2\equiv 0\pmod {14}.
\end{align*}
From $s+2t \equiv 0 \pmod{15}$ and $3s+t \equiv 0 \pmod{14}$, we
obtain
$$5s = -15k+ 28l, \ 5t = 45k - 14l,$$
for some integers $k,l$. Thus, $l =5m$ for some integer $m$, which gives us
$$s = -3k + 28m, t = 9k- 14m.$$
Thus, any solution $(u,v)$ of the system of congruences \eqref{eq0} is of the form $u=10+3k-28m, v = 21 -9k+14m$.

\noindent
If this solution satisfies the GCD restrictions
$$(u,15)=5, (u,14)=2, (v,15)=3, (v,14)=7,$$
then $10|u, 21|v$. In particular, $7|k$; say, $k=7K$. So,
$$u = 10+21K-28m, v = 21-63K+14m.$$
From $21|v$, we get $3|m$; write $m=3M$. Then,
$$u=10+21K-84M$$
which gives, by the condition $10|u$, that $K-4M = 10N$ for some
$N$. Hence,
$$u = 10+21(K-4M)= 10+210N \equiv 10 \pmod{210};$$
$$v =21-63K+42M= 21-63(10N+4M)+42M = 21 -630N -210M \equiv
21 \pmod{210}.$$

\section{Restricted linear congruences over $\mathbb{F}_q[t]$}

We show that the above results for integers can be generalized in an
appropriate way to the case of $\mathbb{F}_q[t]$ where
$\mathbb{F}_q$ is a finite field of cardinality $q$. 

Let $k$ and $n$ be arbitrary positive integers and $A_{ij} \in \mathbb{F}_q[t]$
for $1 \leq i \leq k$, $1 \leq j \leq n$. For $H_{1}, \dots, H_{k}
\in \mathbb{F}_q[t]$ of positive degrees and $B_{1}, \dots, B_{k}
\in \mathbb{F}_q[t]$ be arbitrary. Then, consider the system of congruences over $\mathbb{F}_q[t]$:
\begin{align}\label{congs eq ff}
    A_{11}X_{1}+A_{12}X_{2}+\dots+A_{1n}X_{n} &\equiv B_{1} \pmod{H_{1}},\nonumber\\
    A_{21}X_{1}+A_{22}X_{2}+\dots+A_{2n}X_{n} &\equiv B_{2} \pmod{H_{2}},\nonumber\\
    \cdots \cdots \cdots \cdots & \cdots \cdots\nonumber\\
    A_{k1}X_{1}+A_{k2}X_{2}+\dots+A_{kn}X_{n} &\equiv B_{k} \pmod{H_{k}}.
\end{align}
We consider the existence of solutions $X_i \in \mathbb{F}_q[t]$,
and count the number of solutions modulo the
ideal generated by $H=H_{1}\cdots H_{k}$. For every $H \in \mathbb{F}_{q}[t]$, the function $|H|=q^{\deg{H}}=|\mathbb{F}_q[t]/(H)|$
denotes the absolute value function on $\mathbb{F}_{q}[t]$.
We shall use the notations
$(H_{1}, \dots, H_{k})$ and $[H_{1}, \dots, H_{k}]$  to denote the
gcd and lcm of the polynomials $H_{1}, \dots, H_{k}$ respectively.
Given non-zero polynomials $H_{1}, \ldots, H_{k}$, both the $(H_{1}, \dots, H_{k})$ and $[H_{1}, \dots, H_{k}]$ are well-defined upto units. Thus, we say that any two polynomials $H_{1}\,\textit{and}\, H_{2}$ are coprime, if $(H_{1}, H_{2})$ is a unit; for convenience, we write $(H_{1}, H_{2})=1$. We prove the following theorem.

\begin{theorem}\label{soln of congs thm ff}
Let $H_{1}, \dots, H_{k} \in \mathbb{F}_q[t]$ be pairwise co-prime
non-constant polynomials, and $H=H_{1}\cdots H_{k}$. The system of
congruences represented by \eqref{congs eq ff} has a solution
$(X_{1}, \dots, X_{n}) \in \mathbb{F}_{q}[t]^{n}$ if, and only if,
for each $i \leq k$, $L_{i} \mid B_{i}$, where $L_{i}= (A_{i1},
\dots, A_{in}, H_{i})$. Further, if this condition is satisfied,
then there are $|H|^{n-1} \prod_{i=1}^{k} |L_{i}|$ solutions.
\end{theorem}

\noindent Moreover, we may define the analogues of Euler's $\varphi$
function, the M\"{o}bius function and Ramanujan sums for
$\mathbb{F}_q[t]$ and prove results analogous to the ones proved in
the earlier sections. In what follows, we will discuss this in some detail though the proofs themselves are completely analogous to
the above ones.

\vspace{1mm}\noindent
For a non-constant polynomial $H$, define $\varphi(H)$
to be the number of polynomials $Q \in \mathbb{F}_q[t]$ of degree
less than $\deg(H)$ which are coprime to $H$. Then, we shall prove:

\begin{theorem}\label{sing lin thm ff}
Let $H$ be a non-constant polynomial as before, and let $A, B \in
\mathbb{F}_q[t]$ be arbitrary non-zero polynomials. Then, the
congruence $AX \equiv B \pmod {H}$ has a solution $X \in
\mathbb{F}_q[t]$ satisfying $(X, H ) = T$ if and only if $T | (B,
H)$ and $\left(A, \frac{H}{T}\right) = \left(\frac{B}{T},
\frac{H}{T}\right)$, where $T\neq 0$. Furthermore, if these
conditions are satisfied, then there are exactly
\[
\frac{\varphi(\frac{H}{T})}{\varphi(\frac{H}{DT})}
\]
solutions modulo $H$, where $D=(A, \frac{H}{T})
=(\frac{B}{T},\frac{H}{T})$.
\end{theorem}

\noindent Finally, a notion of Ramanujan sums for $\mathbb{F}_q[t]$
was defined by Z. Zheng \cite{ZZ18}. For $A,B \in \mathbb{F}_q[t]$, the
Ramanujan sum $\eta(A,B)$ is defined in Section \ref{basic def ff} (see \eqref{rama sum ff}) and we
prove:

\begin{theorem}\label{without coeff thm ff}
Let $H$ be a non-constant polynomial, and let $H_i|H\,(1\le i\le
k)$. Then, for any $B \in \mathbb{F}_{q}[t]$, the number of
solutions modulo $H$ of the linear congruence
$$X_1+\cdots+X_k\equiv B \pmod H, \ \text{with} \ (X_i,H)=H_i\, (1\le i\le k),$$
is given by

\[
\dfrac{1}{|H|}\sum_{D|H}\eta(B,D)\prod_{i=1}^{k}\eta(H/D,H/H_i).
 \]
\end{theorem}

As an application of Theorem \ref{sing lin thm ff} and Theorem \ref{without coeff thm ff}, we prove the following Theorem:
\begin{theorem}\label{res lin cong ff}
Let $H_{1}, \dots, H_{k}$ be pairwise coprime polynomials in $\mathbb{F}_q[t]$ and
$H=H_{1}\cdots H_{k}$. Let $T_{ij} \mid H_{i}$ (for $1 \leq i \leq
k$, $1 \leq j \leq n$). Then the number of solutions of the system
of congruences represented by \eqref{congs eq ff} with the
restrictions $(X_{j}, H_{i})=T_{ij}$ for $1 \leq i \leq k$, $1 \leq
i \leq n$, is given by
\[
\frac{1}{|H|}\prod_{j=1}^{n} \frac{\varphi(\frac{H}{T_{j}})}{\varphi(\frac{H}{T_{j}D_{j}})}
\sum_{D \mid H} \eta{(B,D)} \prod_{l=1}^{n}\eta\bigg(\frac{H}{D}, \frac{H}{T_{l}D_i}\bigg),
\]
where $T_j= \prod_{i=1}^{k}T_{ij}$
and $D_j= \prod_{i=1}^{k}D_{ij}$ with $D_{ij}=(A_{ij}, \frac{H_i}{T_{ij}})$
for $1\le i\le k$ and $1\le j\le n$. Here $B$ represents the unique
solution of the following system of congruences
\begin{align*}
X &\equiv B_{i} \pmod{H_{i}}, \ 1 \leq i \leq k.
\end{align*}
\end{theorem}

\section{Ramanujan Sums and DFT}\label{basic def ff}
In \cite{ZZ18}, Z. Zheng introduced the concept of the Ramanujan sum
within the context of $\mathbb{F}_q[t]/\langle H\rangle$, where
$H\in \mathbb{F}_q[t]$ represents a fixed non-constant polynomial in
$\mathbb{F}_q[t]$. Let $A$ be any polynomial such that $A\equiv
a_{m-1}t^{m-1}+\cdots + a_1x+a_0\pmod H$, and $\tau$ be a function
from $\mathbb{F}_q[t]/\langle H\rangle$ to $\mathbb{F}_q$ defined by
$\tau(A)=a_{m-1}$. Then, $\tau$ is an additive function modulo $H$ on
$\mathbb{F}_q[t]$, that is, for any $A$ and $B$ in $\mathbb{F}_q[t]$,
we have $\tau(A+B)=\tau(A)+\tau(B)$, and if $A\equiv B\pmod H$
then $\tau(A)=\tau(B)$. In particular $\tau(A)=0$ whenever $H|A$. In
general $\tau_{G}(A)=\tau(GA)$, then $\tau_G(A)$ is also an additive
function modulo $H$.

\vspace{2mm} Next, let $\lambda$ be a fixed non-principal character
on $\mathbb{F}_{q}$; for example, one may choose $\lambda(a) =
e(\frac{tr(a)}{p})$ for $a\in \mathbb{F}_{q}$, where $tr$ is the
trace map from $\mathbb{F}_q$ to $\mathbb{F}_p$, and $e(x)= e^{2\pi
ix}$. We define a complex valued function $E(G,H)$ on
$\mathbb{F}_q[t]$ by
$$E(G,H)(A)=\lambda(\tau_G(A)).$$ It is easy to see that $E(G,H)$ is an additive character modulo $H$ on $\mathbb{F}_q[t]$ and
\begin{equation*}
E(G,H)(A)=E(A,H)(G).
\end{equation*}

\noindent Moreover, the following lemma shows that any additive
character modulo $H$ on $\mathbb{F}_q[t]$ is of the form $E(G, H)$
for some $G \in \mathbb{F}_q[t]$.
\begin{lemma}(Lemma 2.1, \cite{ZZ18})
For any $\psi$, an additive character modulo $H$ on
$\mathbb{F}_q[t]$, there exists a unique polynomial $G$ in
$\mathbb{F}_q[t]$ such that $\psi=E(G, H )$, and $\deg{G} <
\deg{H}$.
\end{lemma}
The next lemma gives the orthogonal relation for the additive characters
modulo $H$.
\begin{lemma}\label{ortho of E lem}(Lemma 2.3, \cite{ZZ18})
Suppose $A \in \mathbb{F}_{q}[t]$, then we have
\begin{equation*}
   \sum_{G \pmod H} E(G, H )(A)= \begin{cases}
     |H|, & \text{if } H \mid A,\\
     0, & \text{otherwise}.
    \end{cases}
\end{equation*}
\end{lemma}

\textit{In what follows, $D \mid H$ means $D$ is a monic divisor of
$H$ and $\sum_{D \mid H}$ means $D$ extending over all of monic
divisors of $H$.}

\vspace{1mm} The polynomial Ramanujan sum modulo $H$ on
$\mathbb{F}_q[t]$ is given by (see Section 4.1,  \cite{CL47})
\begin{equation}\label{rama sum ff}
\eta(G,H)=\sum_{\substack{A \pmod H\\(A, H)=1}}E(G,H)(A)
\end{equation}
where the summation extends over a complete residue system modulo
$H$ in $\mathbb{F}_q[t]$. Before we discuss the properties of the
Ramanujan sum, let us introduce the following definition.
\begin{definition} The M\"{o}bius function $\mu$ from $\mathbb{F}_q[t]$ to
$\{0,1,-1\}$ as follows: $\mu(0)=0$ and for $H \neq 0$
\begin{equation*}
   \mu\left({H}\right)= \begin{cases}
    (-1)^h ,& \text{if } H=P_1 \cdots P_h\, \text{where}\, P_j'\text{s} \,\text{are pairwise non-associate} \\
    & \quad \quad \quad \text{irreducible polynomials},\\
     0, & \text{if } H \,\text{is not square free},\\
     1, & \text{if } H\in \mathbb{F}_q^{*},
    \end{cases}
\end{equation*}
where $\mathbb{F}_q^{*}$ denotes the set of units of
$\mathbb{F}_{q}$.
\end{definition}
Clearly, $\mu(H )$ is a multiplicative function on
$\mathbb{F}_q[t]$. Furthermore, it satisfies the following identity:
\begin{equation*}
   \sum_{D|H}\mu\left({D}\right)= \begin{cases}
     1, & \text{if } \deg{H}=0\\
     0, & \text{if } \deg{H}\ge 1
    \end{cases}
\end{equation*}
\begin{definition} (Euler's totient function) If $H$ is not constant, $\varphi(H)$ is
defined to be the number of polynomials of degree less than
$\deg{H}$ that are coprime to $H$. Put also $\varphi(0) = 0$ and $\varphi(a) = 1$ for all $a\in \mathbb{F}_q^*$. Further, it is easy to verify that $\varphi(H)=\sum_{D \mid H}\mu(H/D)|D|$.
\end{definition}
 
It is easy to see that $\varphi$ is a multiplicative function. Now,
we proceed to list several properties of the polynomial Ramanujan
sum that are analogous to the classical Ramanujan sum.
\begin{itemize}
\item[(i)] For any polynomials $G,H_1,H_2$, we have $$\eta (G, H_1H_2)=\eta(G,H_1)\eta(G,H_2)\quad \text{if} \, (H_1,H_2)=1.$$
\item[(ii)] For any polynomials $G,H$ with $H\neq 0$, we have
\begin{equation*}
\eta(G, H)= \displaystyle\sum_{D|(H,
G)}|D|\mu\left(\frac{H}{D}\right).
\end{equation*}
\item [(iii)]  For any polynomials $G,H$ with $H\neq 0$, we have
\begin{equation*}
\eta(G,H)=\frac{\varphi(H)\mu(N)}{\varphi(N)}, \ \text{where} \ N=\frac{H}{(G,H)}.
\end{equation*}
\end{itemize}

\vspace{2mm} A complex-valued function defined on $\mathbb{F}_q[t]$
is called an arithmetic function on $\mathbb{F}_q[t]$ (or) simply an
arithmetic function if $f(aA)=f(A)$ for any $a \in
\mathbb{F}_{q}^{*}$. More generally, an arithmetic function of $n$
variables is a function  $f: (\mathbb{F}_q[t])^{n} \rightarrow
\mathbb{C}$ satisfying $f(a_{1}A_{1}, \dots, a_{n}A_{n})=f(A_{1},
\dots, A_{n})$ for any $(a_{1}, \dots, a_{n}) \in
(\mathbb{F}_{q}^{*})^{n}$.

\vspace{1mm} \noindent Let $\mathfrak{F}_{n}$ denotes the set of all
arithmetic function of $n$ variables. If $f, g \in \mathfrak{F}_{n}$,
then their convolution is defined as
\begin{equation*}
(f \ast g)(H_{1}, \dots, H_{n})= \sum_{D_{1} \mid H_{1}, \dots,
D_{n} \mid H_{n}} f(D_{1}, \dots, D_{n}) g\big(\frac{H_{1}}{D_{1}},
\dots, \frac{H_{n}}{D_{n}}\big).
\end{equation*}

The set $\mathfrak{F}_{n}$ forms a ring with point-wise addition and
convolution product with unit element $\mathcal{E}_n(H_1,
H_2,\ldots, H_n)$ defined by
\begin{equation*}
  \mathcal{E}_n(H_1,H_2,\ldots, H_n)= \begin{cases}
     1, & \text{if } H_1, H_2,\ldots, H_n\in \mathbb{F}_q^{*}\\
     0, & \text{othewise}
    \end{cases}
\end{equation*}
\begin{lemma}
An arithmetic function $\delta$ on $(\mathbb{F}_q[t])^n$ is
invertible if and only if $\delta(A_1,\ldots,A_n)\neq 0$ for any
$A_1,\ldots,A_n\in \mathbb{F}_q^{*}$.
\end{lemma}
\begin{lemma}(M\"{o}bius inversion formula)\label{IMF}
If $\Delta$ and $\delta$ are two arithmetic functions of $n$
variables satisfying
\[
\Delta(H_1,\ldots,H_n)=\sum_{D_1|H_1,\ldots,D_n|H_n}\delta(D_1,\ldots,D_n)
\]
then
\[
\delta(H_1,\ldots,H_n)=\sum_{D_1|H_1,\ldots,D_n|H_n}\Delta(D_1,\ldots,D_n)\mu{(H_1/D_1)}\cdots\mu{(H_n/D_n)}
\]
\end{lemma}
\vspace{2mm} An arithmetic function $\Delta$ is said to be periodic
with period $H$ (or $H$-periodic) for some $H \in \mathbb{F}_q[t]$
if $\Delta(H_1)=\Delta(H_2)$ for every $H_1,H_2$ satisfying
$H_1\equiv H_2\pmod H$.  For an $H$-periodic arithmetic function
$\Delta$, its discrete (finite) Fourier transform (DFT) is defined
to be the function
\begin{align*}
\widehat{\Delta}{(G)}=\sum_{A \pmod H}\Delta(A)E(G,H)(A).
\end{align*}
Also, the inverse discrete Fourier transform (IDFT) of $\Delta$ is given by
\begin{align*}
\Delta(G)=\frac{1}{|H|}\sum_{A \pmod
H}\widehat{\Delta}{(A)}E(-G,H)(A).
\end{align*}
We say that an
arithmetic function $\Delta$ is $H$-even if $f(G)=f((G,H))$, for
every $G \in \mathbb{F}_q[t]$. Clearly, if a function $\Delta$ is
$H$-even then it is $H$-periodic. Further, for a $H$-even function
$\Delta$, we have
\begin{align*}
\widehat{\Delta}{(G)}&=\sum_{D|H}\Delta(D)\sum_{\substack{A \pmod H \\(A, H/D)=1}}E(G,H/D)(A)\\
&= \sum_{D|H}\Delta(D)\eta(G,H/D).
\end{align*}

\noindent The Cauchy convolution of two $H$-periodic function
$\Delta$ and $\delta$ is defined by
\[
(\Delta\otimes\delta)(G)=\sum_{\substack{A,B\pmod H \\ A+B \equiv G
\pmod H}}\Delta(A)\delta(B)=\sum_{ A \pmod H} \Delta(A)\delta(G-A).
\]
Therefore, we have
\begin{align*}
\widehat{(\Delta\otimes\delta)}{(G)}&=\sum_{A \pmod H}\sum_{B  \pmod H}\Delta(B)\delta(A-B)E(G,H)(A)\\
&= \sum_{B \pmod H}\Delta(B)E(G,H)(B)\sum_{A \pmod H}\delta(A)E(G,H)(A)\\
&= \widehat{\Delta}{(G)}\widehat{\delta}{(G)}.
\end{align*}
Similarly, we can define the Cauchy convolution of a finite number
of $H$-periodic functions. Next, for every
$B\in\mathbb{F}_q[t]/\langle H\rangle$, we define a $H$-periodic
function as
\begin{equation*}
 \rho_{H,H_1}(B)=\begin{cases}
  1, & \text{if} \,(B,H)=H_1\\
  0, & \text{if}\, (B,H)\neq H_1.
 \end{cases}
\end{equation*}
Then, it is easy to see that
\begin{equation*}
\widehat{\rho_{H,H_1}}(B)=\eta(B,H/H_1).
\end{equation*}
\vspace{2mm}
\noindent
Let $H_{1}, \dots, H_{n}, A \in
\mathbb{F}_q[t]$. For any $H \in \mathbb{F}_q[t]$ such that $[H_{1},
\dots, H_{n}] \mid H$, we define
\begin{equation*}
I(A; H_{1}, \dots, H_{n}): = \frac{1}{|H|} \sum_{C \pmod H} \eta(C,
H_{1}) \cdots \eta(C, H_{n})E(A,H)(C),
\end{equation*}
and
\begin{equation*}
J(A; H_{1}, \dots, H_{n})= \begin{cases}
    \frac{|H_{1}| \cdots |H_{n}|}{|[H_{1}, \dots, H_{n}]|} ,& \text{if }   \frac{H}{[H_{1}, \dots, H_{n}]} \mid A\\
     0, & \text{otherwise}.
    \end{cases}
\end{equation*}

\begin{lemma}\label{MI of E ff}
For any $H_{1}, \dots, H_{n}, A \in \mathbb{F}_q[t]$, we have
\begin{equation*}
\sum_{D_{1} \mid H_{1}, \dots, D_{n} \mid H_{n}} I(A;D_{1}, \dots,
D_{n})= J(A; H_{1}, \dots, H_{n}).
\end{equation*}
\end{lemma}
The proof of Lemma \ref{MI of E ff} follows similarly to the Lemma \ref{MI of E} in Section \ref{lcm and gcd sec}.

\subsection{Proof of the Theorems}
By using the above notations, the proof of Theorem \ref{sing lin thm ff} and Theorem \ref{without coeff thm ff} follows similarly to the proof of Theorem 3.1 and Theorem 3.4 in \cite{BBVRL17}, respectively. The proof of Theorem \ref{soln of congs thm ff} is similar to the proof of Theorem \ref{soln of congs thm}, but with the use of Lemma \ref{MI of E ff} instead of Lemma \ref{MI of E} (see Section \ref{lcm and gcd sec}). The proof of Theorem \ref{res lin cong ff} follows similarly to the proof of Theorem \ref{Theorem 3} (see Section \ref{soln of lin cong}).
\subsection{Some examples}
Here, we examine the solution count given by Theorem \ref{soln of congs thm ff} and Theorem \ref{res lin cong ff} through some examples and also verify them by the first principle.\\
\noindent {\bf Example for Theorem \ref{soln of congs thm ff}.}\\
Consider the solutions of the congruence
$$t f(t) + t^2 g(t) \equiv 2t+t^3~~(mod~~t^4).$$
We are looking for solutions $f,g$ in $\mathbb{F}_q[t]/(t^4).$ It follows from Theorem \ref{soln of congs thm ff} that the number of solutions is $|t^4||(t,t^2,t^4)| =|t^4||t| = q^5$.

\noindent
To check this by first principle, we note that $(2,t)$ is a solution. So, if
$(f,g)$ is an arbitrary solution with $\deg(f), \deg(g) \leq 3$, then
$$t (f(t) - 2) + t^2 (g(t) - t) = t^4 h(t),$$
where $h(t) = h_0+h_1t+ \cdots \in \mathbb{F}_q[t]$. So, writing
$g(t)= b_0+b_1 t+b_2 t^2+b_3 t^3$, we have
$$f(t)  =  2 - b_0 t + (1-b_1) t^2 + (h_0 - b_2) t^3 \pmod{t^4}.$$
So, the coefficient of $t^3$ for $f$ is arbitrary, and others are
determined. The coefficients of $g$ are arbitrary. Hence, the number
of solutions is $q^5$.\\
Note that $(f,g)$ varies in $\mathbb{F}_q[t]/(t^4) \times
\mathbb{F}_q[t]/(t^4)$ which has cardinality $q^8$.

\vspace{2mm}
\noindent {\bf Example for Theorem \ref{res lin cong ff}.}\\
 Let $q$ be an odd prime.
Consider the following system of congruences
\begin{align}\label{eq3}
&X_1+(1+t)X_2\equiv 3t +1\pmod {t^2} \nonumber\\
&X_1+ X_2\equiv -2\pmod {t+1},
\end{align}
with the restrictions
$$(X_1,t^2)=1, (X_1,t+1)=1, (X_2,t^2) = t, (X_2,t+1)=1.$$
Then we observe that $(2t+1, t)$ is one solution of the system
\eqref{eq3}. We shall prove that the number of solutions for
$X_1,X_2$ modulo $t^2(t+1)$ is $(q-1)(q-2)$. We first use Theorem \ref{res lin cong ff} to determine the number and verify it by first principles as
a check.\\
In the notation of the theorem, then $H=t^2(t+1)$, $T_{11}=1,
T_{12}=t, T_{21}=1, T_{22}=1$, so $T_1=1, T_2 =t$ and $D_{ij}=1$ for
$1\le i,j\le 2$, hence $D_i=1$ for $1\le i\le 2$.  To find all the
solutions of \eqref{eq3}, we need only consider the following
congruences
\[
Y_1+Y_2\equiv B \pmod {t^2(t+1)},
\]
satisfying $(Y_1, t^2(t+1))=1$ and $(Y_2, t^2(t+1))=t$, where $B =
3t+1$ is the unique solution of the congruences $B \equiv 3t+1\pmod
{t^2}, B \equiv -2 \pmod{(t+1)}$. Recall that $\varphi(H)=\sum_{D|H}\mu(D)\big|\frac{H}{D}\big|$ and $\eta(G,H)=\frac{\varphi(H)\mu(N)}{\varphi(N)}$, where $N=\frac{H}{(G,H)}$.\\
We first compute the number of solutions using Theorem \ref{res lin cong ff}. For any
prime $q$, the number of solutions of the system of congruences \eqref{eq3} is equal to
\begin{align*}
&\frac{1}{|t^2(t+1)|}\times\frac{\varphi({t^2(t+1)}\varphi(t(t+1))}{\varphi({t^2(t+1)}\varphi(t(t+1))}\times\\
&\sum_{D|t^2(t+1)}\eta(3t+1,D)\prod_{j=1}^{2}\eta\left(\frac{t^2(t+1)}{D},\frac{t^2(t+1)}{T_jD_j}\right)\\
\quad
&=\frac{1}{q^{3}}\sum_{D|t^2(t+1)}\eta(3t+1,D)\prod_{j=1}^{2}\eta\left(\frac{t^2(t+1)}{D},\frac{t^2(t+1)}{T_jD_j}\right).
\end{align*}

Now, we compute the product of the Ramanujan sums for each divisor $D$ of $t^{2}(t+1)$ in the following table:
\begin{center}
\begin{tabular}{ | m{4em} | m{1.5cm}| m{2cm} |m{2cm} | m{2cm}| m{2cm}| m{1.5cm}| }
  \hline
  Sl. no& $D=$ & $\eta(3t+1,D)$ & $\eta\left(\frac{t^2(t+1)}{D}\right)$ & $\eta\left(\frac{t^2(t+1)}{T_jD_j}\right)$ & Product \\
  \hline
  1 & $1$ & $1$ & $q(q-1)^2$ & $(q-1)^2$& $q(q-1)^4$ \\
  \hline
  2 & $t$ & $-1$ & $-q(q-1)$ & $(q-1)^2$& $q(q-1)^3$ \\
  \hline
  3 & $t+1$ & $-1$ & $-q(q-1)$ & $-(q-1)$& $-q(q-1)^2$\\
  \hline
  4 & $t^2$ & $0$ & $-$ & $-$& $0$\\
  \hline
  5 & $t(t+1)$ & $1$ & $q$ & $-(q-1)$& $-q(q-1)$\\
  \hline
  6 & $t^2(t+1)$ & $0$ & $-$ & $-$& $0$\\
  \hline
\end{tabular}
\end{center}
Hence, the number of solutions is equal to
$$\frac{1}{q^3}[q(q-1)^4+q(q-1)^3-q(q-1)^2-q(q-1)]=(q-1)(q-2).$$

\vspace{2mm}
\noindent Let us compute the number of solutions by first principles now. As we observed above, we are looking for solutions of
\[
Y_1+Y_2\equiv 3t+1 \pmod{t^2(t+1)},
\]
satisfying $(Y_1, t^2(t+1))=1$ and $(Y_2, t^2(t+1))=t$. For $Y_2=at$ (with $1 \leq a \leq q-1$), the unique solution $Y_1
\equiv 3t+1-at$ modulo $t^2(t+1)$ satisfies the restriction $(Y_1,
t^2(t+1))=1$ if and only if, $a \neq 0, 2$. Therefore, $(3t+1-at, at)$ are solutions for $1 \leq a \leq q-1$ and $a \neq 0,2$. These give $q-2$ solutions.

\noindent
For $Y_2 = at^2+bt$ with $a \neq 0$, the unique solution $Y_1 =
3t+1-at^2-bt$ modulo $t^2(t+1)$ satisfies $(Y_1, t^2(t+1))=1$ if
and only if, $b \neq 0, a, a+2$. Note that when $a = q-2$, the
restrictions $b \neq 0$ and $b \neq a+2$ coincide. Therefore, the
solutions are $(3t+1-at^2-bt, at^2+bt)$ with $a \neq 0, q-2$ and  $b
\neq 0, a, a+2$; and $(3t+1-(q-2)t^2-bt, (q-2)t^2+bt)$ with $b \neq 0,a$. Thus, the total number of solutions is
$$(q-2) + (q-2)(q-3) + (q-2) = (q-2)(q-1).$$
\subsection*{Acknowledgment}\

This work was done in June 2023. The first and second authors would like to thank the Indian Statistical Institute, Bangalore centre for providing an ideal environment to carry out this work. The second author expresses his gratitude to NBHM for financial
support during the period of this work. 
\bibliographystyle{plain}    

\end{document}